\documentclass{amsart}
\usepackage{amsmath}
\usepackage{paralist}
\usepackage{amssymb}
\usepackage{amsthm}
\usepackage{amscd}
\usepackage{graphicx,mathrsfs}
\usepackage[colorlinks=true]{hyperref}
\hypersetup{urlcolor=blue, citecolor=red}

\numberwithin{equation}{section}
\newtheorem{Thm}{Theorem}[section]
\newtheorem{Lem}{Lemma}[section]
\newtheorem{Prop}{Proposition}[section]

\newtheorem{Cor}{Corollary}[section]

\newtheorem{Conj}{Conjecture}[section]
\newcommand{\R}{\mathbb{R}}

\newcommand{\Z}{\mathbb{Z}}
\newcommand{\scrL}{\mathscr{L}}
\newcommand{\calL}{\mathcal{L}}

\begin{document}
\title[Properties of fractional p-Laplace]{Properties of fractional p-Laplace equations with sign-changing potential}
\author{Yubo Duan}
\address{School of Mathematical Sciences\\ Nankai University\\ Tianjin 300071 China}
\email{1120190017@mail.nankai.edu.cn}
\thanks{Acknowledgements: This work is supported by the NSFC under the grands 12271269 and supported by the Fundamental Research Funds for the Central Universities.}
\author{Yawei Wei}
\address{School of Mathematical Sciences and LPMC\\ Nankai University\\ Tianjin 300071 China}
\email{weiyawei@nankai.edu.cn}

	
	
%	\centerline{\scshape Huiying Zhao}
%	\medskip
%	{\footnotesize	\centerline{Nankai University, Tianjin 300071, China}}
%\bigskip
	
\begin{abstract}
In this paper, we consider the nonlinear equation involving the fractional p-Laplacian with sign-changing potential. This model draws inspiration from De Giorgi Conjecture. There are two main results in this paper. Firstly, we obtain that the solution is radially symmetric within the bounded domain, by applying the moving plane method. Secondly, by exploiting the idea of the sliding method, we construct the appropriate auxiliary functions to prove that the solution is monotone increasing in some direction in the unbounded domain. The different properties of the solution in bounded and unbounded domains are mainly attributed to the inherent non-locality of the fractional p-Laplacian.
	\end{abstract}

	\maketitle	
\section{Introduction}
In this paper, we study the nonlinear equation involving fractional p-Laplace operator with sign-changing potential,
\begin{equation} \label{SUM}
\left\{
 \begin{aligned}
 &(-\Delta)_{p}^{s}u(x)+a(x)|u|^{p-2}u(x) = f(u(x)), \,\,\, \text{in} \,\,\, \Omega,\\
 &u>0,\qquad   \textup{in}\,\,\,  \Omega, \\
 &u= 0 , \qquad \textup{on}\,\,\, \Omega^{c},
 \end{aligned}
 \right.
\end{equation}
where $\Omega$ is either a bounded or an unbounded domain.
The fractional p-Laplacian is a non-local and nonlinear operator, which is defined as follows
\begin{equation}\label{fpl}
\begin{aligned}
(-\Delta)_{p}^{s}u(x)&:=C_{n,s,p}P.V.\int_{\R^{n}}\frac{|u(x)-u(y)|^{p-2}(u(x)-u(y))}{|x-y|^{n+sp}}dy\\
&=C_{n,s,p}\lim_{\varepsilon \rightarrow 0}\int_{\R^{n}\setminus B_{\varepsilon}(x)}\frac{|u(x)-u(y)|^{p-2}(u(x)-u(y))}{|x-y|^{n+sp}}dy,
\end{aligned}
\end{equation}
where $0<s<1$, $2< p <\infty$, the constant $C_{n,s,p}$ depends on $n$, $s$, and $p$, and $P.V.$ is the Cauchy principal value. The potential function $a(x)$ may change signs here.
Denote simply
\begin{equation} \label{LP}
\scrL_{p}u:=(-\Delta)_{p}^{s}u+a(x)|u|^{p-2}u.
\end{equation}
%which is positive definite operator, i.e. the first eigenvalue $\lambda_{1,sp}(a)$ of  $\scrL_{p}$ is positive in bounded domains. According to the definition and existence of the first eigenvalue of $(-\Delta)_{p}^{s}$ in \cite{LBJL}, here we define that
%\begin{equation} \label{lambda1}
%\lambda_{1,sp}(a):=\inf \left \{ \frac{1}{2} \int_{\R^{n}} \int_{\R^{n}} \frac{|u(x)-u(y)|^{p}}{|x-y|^{n+sp}}
%dydx + \int_{\Omega} a(x)|u|^{p} dx: \| u \| _{L^{p}}=1\}\right.
%\end{equation}
In order to make the integral in \eqref{fpl} well defined, we require that
$$u\in C^{1,1}_{loc}(\Omega)\cap \mathcal L_{sp}, $$
with
$$\mathcal L_{sp}:=\{ u\in L^{p-1}_{loc}|\int_{\R^{n}}\frac{|u(x)|^{p-1}}{1+|x|^{n+sp}}dx <\infty \}.$$

%In particular, when $p=2$, $(-\Delta)_{p}^{s}$ becomes the fractional Laplacian $(-\Delta)^{s}$. In \cite{LBJ} the authors obtain that as $s \rightarrow 1$, the fractional Laplacian converges to the regular Laplacian. One also can show the same conclusion referring to \cite{LC}
%$$(-\Delta)^{s}_{p}u(x)\rightarrow -\Delta_{p}u(x):= -div(|\nabla u(x)|^{p-2}\nabla u(x)).$$

The existence and regularity of solutions of the equations involving the operator $(-\Delta)_{p}^{s}$ preserves fruitful results. In \cite{SKM}, Mosconi, et al, obtained nontrivial solutions of the Brezis-Nirenberg problem concerning the fractional p-Laplace operator. Then, the operator $(-\Delta)_{p}^{s}$ with the potential term, like the form $\scrL_{p}$, has been much concerned. Radulescu, et al, in \cite{DXZ} investigated the existence of weak solutions for a perturbed nonlinear elliptic equation driven by $(-\Delta)_{p}^{s}$, with the positive potential function. Ambrosio \cite{VA}, Zhang, and collaborators \cite{ZTZ} derived the existence of infinitely many weak solutions for the equation involving $\scrL_{p}$ with sign-changing potential function $a(x)$. Chen, Liu \cite{LC}, and  Li \cite{Liz} have obtained the uniform H\"older norm estimate of the solution of the fractional p-Laplace equation on the whole space, if the domain satisfies the uniform two-sided ball condition. These conclusions ensure the existence and H\"older regularity of the solution of the fractional p-Laplace equation \eqref{SUM} that we consider.

%The regularity for fully nonlinear integro-differential operator also has been wildly studied. Caffarelli and Silvestre \cite{LL} earlier obtain the H\"older regularity results for nonlinear integro-differential equations. Brasco, et al in \cite{LEA} give some H\"older estimates for fractional p-Laplace equation in bounded domains. Iannizzotto, et al \cite{ASM} develop the global H\"older regularity (up to the boundary) for fractional p-Laplace equation. In \cite{LY} Jin, Li get boundary regularity for the fractional p-Laplcian. These results ensure the regularity of the problem \eqref{SUM}.

The moving plane method has been widely applied to study fractional partial differential equations. This method was initially invented by Alexanderoff in the early 1950s. Later, it was evolved by Serrin \cite{SJ}, Gidas, Ni, and Nirenberg \cite{GNN1,GNN2}, Caffarelli, Gidas, and Spruck \cite{CGS}. Recently, Chen and Li in \cite{CLC1} developed this method so that it can be applied to various semi-linear equations involving the fractional Laplacian to explore symmetry, monotonicity and non-existence of the solutions. In \cite{CLC},  Chen and Liu also introduced maximum principles and take the moving plane method to establish symmetry of solutions of the nonhomogeneous fractional p-Laplace equation with $(-\Delta)_{p}^{s}$, both in the bounded domain and the whole space. Dai, Fang, and Qin \cite{DFQ} explored fractional order Hartreen equations by applying this method. In the present paper, we also make use of the moving plane method to derive that the solution of the equation \eqref{SUM} is radially symmetric within a bounded domain.

While the sliding method is similar to the moving plane method, which can be applied to the equations without decaying conditions at infinity, to discover the properties of the solutions. The classical sliding method was invented by Berestycki and Nirenberg\cite{BN2}-\cite{BN3}, which is for establishing qualitative properties of solutions for partial differential equations involving the Laplace operator, such as symmetry, monotonicity, uniqueness, etc. Chen, Wu \cite{WLC1}, and Liu \cite{LC} have developed the sliding method for the fractional p-Laplace equation with $(-\Delta)_{p}^{s}$. Differing slightly from the moving planes method, the idea of the sliding method is to compare the solution with its translation rather than its reflection, and the key point of the sliding process is to find a maximizing sequence near infinity. For the equation \eqref{SUM} in the unbounded domain, we prove that the solution is monotone increasing in some direction by using the idea of the sliding method. Here the classical sliding method in \cite{LC} has been modified through a series of appropriate auxiliary functions.

The main results of this paper are as follows.

\begin{Thm} \label{Them1}
Let $\Omega=B_{1}(0)$, $p>2$ and $u \in C_{loc}^{1,1}(\Omega) \cap \mathcal{L}_{sp}$  be a positive bounded solution of \eqref{SUM}. Let the potential function $a(x)$ be bounded from below, and $f$ be Lipschitz continuous with respect to $u$. Then the solution $u(x)$ is radially symmetric and decreasing from the origin.
\end{Thm}

%Follow the same argument of the proof of Theorem \ref{Them1}, we have the following Corollary.
%\begin{Cor} \label{cor1}
%Under the same conditions for $u$ and $f$ in Theorem \ref{Them1}, if the potential function $a(x)$ is positive, the same conclusion as Theorem \ref{Them1} also is obtained.
%%Let $\Omega=B_{1}(0)$, $p>2$ and $u \in C_{loc}^{1,1}(\Omega) \cap \mathcal{L}_{sp}$  be a bounded solution of \eqref{SUM}. Let the potential $a(x) >0$, and $f$ be Lipschitz continuous. Then the solution $u(x)$ is radially symmetric and decreasing from the origin.
%\end{Cor}
We recall the fact that the solution of the fractional p-equation has the uniform H\"older estimate in $\R^{n}$ with the domain $\Omega$ satisfying the uniform two-sided ball condition, which includes the exterior and interior sphere conditions, the details are referred to \cite{LC},\cite{Liz}.

\begin{Thm}\label{Them2}
Let $\Omega$ be an unbounded domain in $x_{n}$-direction with the uniform two-sided ball condition, $p>2$ and $0<s<1$. Let the solution $u \in C_{loc}^{1,1}(\Omega) \cap \mathcal{L}_{sp}$  be a positive bounded solution of \eqref{SUM} with the upper bound $\mu>0$.  Let $f(u)$ be Lipschitz continuous with respect to $u$, and satisfy the following conditions when $0<t_{0}<t_{1}< \mu$,
\begin{itemize}
\item[$(F_{1})$]  $f(u)>0$ for $u\in(0,\mu)$, and $f(u)\leq 0$ for $u(x)\geq \mu$.
\item[$(F_{2})$]  $f(u)\geq \delta_{0}u$ on $[0,t_{0}]$  for some  $\delta_{0}>0$.
\item[$(F_{3})$]  $f(u)$ is non-increasing for $u\in (t_{1},\mu)$.
\end{itemize}
And let the potential $a(x)$ be such that
\begin{itemize}
 \item[$(A_{1})$] $|a(x)|<\left|\frac{df}{du}\big/\frac{dG}{du}\right|$  for $x\in \Omega$,
      where $G(u):=|u|^{p-2}u$.
 \item[$(A_{2})$] $|a(x)|\thicksim\frac{1}{|x|^{\alpha}}$ with some fixed $\alpha>0$,  for $dist(x,\partial \Omega)>R_{1}>0$, where $R_{1}$ is  fixed and large.
 \end{itemize}

 Then the solution $u$ is strictly monotone increasing in $x_{n}$-direction.
\end{Thm}

The conditions $(F_{1})$ and $(F_{2})$ that are relative to the inhomogeneous term $f(u)$ in \eqref{SUM} are mainly drawn from \cite{LC} and \cite{WLC1}. And the intuitive idea of the conditions $(F_{1})$-$(F_{3})$ comes from the following  De Giorgi Conjecture in \cite{DG}.
\begin{Conj} \label{conj}
 If u is a solution of
 $$-\Delta u=u-u^{3},$$
such that
$$|u|\leq 1 \,\, \mbox{in} \,\, \R^{n}, \,\,\,\lim_{x_{n}\rightarrow\pm\infty}u(x^{\prime},x_{n})=\pm 1 \,~ \mbox{for all}~ \, x^{\prime}\in \R^{n-1},$$
and
$$\frac{\partial u}{\partial x_{n}}>0,$$
then there exists a vector $\mathbf{a}\in \R^{n-1}$ and a function $u_{1}:\, \R \to \R$ such that
$$u(x^{\prime},x_{n})=u_{1}(\mathbf{a}\cdot x^{\prime}+x_{n}).$$
\end{Conj}
%The condition $(A_{1})$ is a natural condition, we hope that the potential term is controlled by the rate of decay of the inhomogeneous term in the following estimating process.

Moreover, if the potential function is non-positive, the strictly monotone increasing property of the solution is also valid under more general conditions. Following the same argument of the proof of Theorem \ref{Them2}, we obtain the following Corollary.
\begin{Cor} \label{cor2}
Under the same conditions for $u$ and $f$ as in Theorem \ref{Them2}, if the potential function $a(x)$ is non-positive and only satisfies the condition $(A_{1})$ in Theorem \ref{Them2}, then the same conclusion as Theorem \ref{Them2}  also holds.
%Let the domain $\Omega$ be unbounded, $p>2$, $0<s<\frac{1}{p}$. Let $u \in C_{loc}^{1,1}(\Omega) \cap \mathcal{L}_{sp}$  be a positive bounded solution of \eqref{SUM}, with the upper bound $\mu>0$. Let the potential function $a(x)$ be non-positive and satisfy $(A_{1})$, $f(u(x))$ be continuous w.r.t. $u$ and satisfy the conditions $(F_{1})$-$(F_{3})$ as in Theorem \ref{Them2},
%Then the solution u is strictly monotone increasing in $x_{n}-$direction.
\end{Cor}

We underline that our results on the properties of the fractional p-Laplace equation are new in the following sense. Firstly, in this paper, the potential function $a(x)$ in our model \eqref{SUM} may change signs.
Comparing the equations involving fractional p-Laplacian in \cite{CLC} and \cite{LC} with the equation in \eqref{SUM}, we can discover that the potential function has little influence on the solution in the bounded domain. However, the sign-changing potential function breaks down the analysis process in the classical sliding method in the unbounded domain. To overcome this obstacle, we build up the condition $(A_{1})$ in Theorem \ref{Them2}, thereby establishing an interrelationship between the inhomogeneous term and the sign-changing term. Moreover, this correlation also indicates that the model in \eqref{SUM} aligns with De Giorgi Conjecture \ref{conj}. Additionally, we infer that the positive solutions of the equation in \eqref{SUM} share the analogous properties with the corresponding asymptotic equation detailed in \cite{LC}. Secondly, while exploiting the classical sliding method, we construct the auxiliary functions, whose domains change along with the maximizing sequence. This construction is new and gives necessary precise estimations in the analysis of the fractional p-Laplace equation in \eqref{SUM}. Furthermore, the idea of the auxiliary functions that we use here can be applied to the monotonicity problems of some relevant equations. Thirdly, according to Theorem \ref{Them1} and Theorem \ref{Them2}, we discover the different properties of the solutions of the same equation in \eqref{SUM} in bounded and unbounded domains. The fact that the geometry of the domain determines the property of the equation in \eqref{SUM} is caused by the non-locality of the fractional p-Laplacian.

 The rest of this paper is organized as follows.  In section $2$, we give the proof of Theorem \ref{Them1}. The radial symmetry of the solution of the equation involving the operator $\scrL_{p}$ in a ball region will be obtained by applying the moving plane method. In section $3$, the preparatory work for the idea of the sliding method will be first given, which includes the existence of a positive lower bound of the solution far away from the boundary in Lemma \ref{lem2-3} and the maximizing sequence near infinity in Proposition \ref{thm2-2}. Then, Theorem \ref{Them2} will be proved finally.

\section{The proof of Theorem \ref{Them1}}
In this section, we study the problem \eqref{SUM} in a bounded domain. By applying the moving plane method, the proof of Theorem \ref{Them1} is given here.

Firstly, we recall an inequality in preparation for the proof.
\begin{Lem}[Appendix.Lemma 5.1\cite{CLC}] \label{xilbd}
 For $G(t)=|t|^{p-2}t$, it is well known that by the mean value theorem, we have
$$G(t_{2})-G(t_{1})=G'(\xi)(t_{2}-t_{1}).$$
Then there exists a constant $c_{1}>0$ such that
 $$|\xi|\geq c_{1}\max\{ |t_{1}|,|t_{2}| \}.$$
\end{Lem}

Let us begin to prove Theorem \ref{Them1}.
\begin{proof}
Let
\begin{equation}
\begin{aligned}
  &T_{\lambda}:=\{x \in \R^{n}|x_{1}=\lambda, \,\,\mbox{for}\,\,\mbox{some}\,\,\lambda \in \R \},\\
  &\Sigma_{\lambda}:=\{x\in \R^{n}|x_{1} <\lambda\},\\
  &x^{\lambda}:=(2\lambda -x_{1},x_{2},\dots ,x_{n}),\\
  &\omega_{\lambda}(x):=u(x^{\lambda})-u(x),\\
  &\Omega_{\lambda}:=\Sigma_{\lambda} \cap B_{1}(0).
\end{aligned}
\end{equation}
where $x^{\lambda}$ is the reflection of $x$ about the $T_{\lambda}.$

We will carry out the proof in two steps. In step 1, we will show that for $\lambda \to -1$, the fact that
\begin{equation}\label{omegald}
\omega_{\lambda}(x) \geq 0, \,\,\,x\in \Sigma_{\lambda}
\end{equation}
 holds, which is the starting point for moving plane method. In step 2, move the plane $T_{\lambda}$ toward right to its limiting position as long as \eqref{omegald} holds. We will show that the limiting position is $\lambda=0$. Since the arbitrariness of direction, we obtain that the solution $u$ is radially symmetric about the origin.

 $Step1$. We claim that as $\lambda \to -1$, \eqref{omegald} holds,

In fact, in $\Sigma_{\lambda} \backslash (\Omega_{\lambda})$, we have $u(x)=0$, $u_{\lambda}(x)\geq 0,$ then
$\omega_{\lambda}(x)\geq 0.$
Assume on the contrary there exists a point $x_{0}\in\Omega_{\lambda}$ such that
\[ \omega_{\lambda}(x_{0})=\min_{\Sigma_{\lambda}} \omega_{\lambda}(x)<0.
\]
Then, by applying the Mean Value Theorem  we have that
\begin{equation} \label{Lul0}
\begin{aligned}
  &\scrL _{p} u_{\lambda}(x_{0})-\scrL _{p} u(x_{0})\\
 =&\ (-\Delta)_{p}^{s}u_{\lambda}(x_{0})-(-\Delta)_{p}^{s}u(x_{0})+a(x_{0})|u_{\lambda}(x_{0})|^{p-2}u_{\lambda}(x_{0})-a(x_{0})|u(x_{0})|^{p-2}u(x_{0})\\
 \leq &\  C_{n,s,p}P.V.\int_{\R^{n}}\frac{G(u_{\lambda}(x_{0})-u_{\lambda}(y))-G(u(x_{0})-u(y))}{|x_{0}-y|^{n+sp}}dy+a(x_{0}) \left(G_{\lambda}(u(x_{0}))-G(u(x_{0}))\right)\\
 =&\ C_{n,s,p}P.V.\int_{\Sigma_{\lambda}}\left(G(u_{\lambda}(x_{0})-u_{\lambda}(y))-G(u(x_{0})-u(y))\right)\left(\frac{1}{|x_{0}-y|^{n+sp}}-\frac{1}{|x_{0}-y^{\lambda}|^{n+sp}}\right)\\
  &\ +C_{n,s,p}P.V.\int_{\Sigma_{\lambda}} \frac{G(u_{\lambda}(x_{0})-u_{\lambda}(y))-G(u(x_{0})-u_{\lambda}(y))+G(u_{\lambda}(x_{0})-u(y))-G(u(x_{0})-u(y))}{|x_{0}-y^{\lambda}|^{n+sp}}dy\\
  &\ +a(x_{0}) \left(G_{\lambda}(u(x_{0}))-G(u(x_{0}))\right)\\
  = &\ C_{n,s,p}P.V.\int_{\Sigma_{\lambda}}\left(G(u_{\lambda}(x_{0})-u_{\lambda}(y))-G(u(x_{0})-u(y))\right)\left(\frac{1}{|x_{0}-y|^{n+sp}}-\frac{1}{|x_{0}-y^{\lambda}|^{n+sp}}\right)\\
   &\ +\omega_{\lambda}(x_{0})\int_{\Sigma_{\lambda}}\frac{G'(\xi(y))+G'(\eta(y))}{|x_{0}-y^{\lambda}|^{n+sp}}dy+a(x_{0}) \left(G_{\lambda}(u(x_{0}))-G(u(x_{0}))\right)\\
   =&\ C_{n,s,p}P.V.(I_{1}+I_{2})+a(x_{0}) \left(G_{\lambda}(u(x_{0}))-G(u(x_{0}))\right),
 \end{aligned}
 \end{equation}
 where $\xi(y)$ and $\eta(y)$ satisfy
 $$t_{1}(y)<\xi(y)<t_{2}(y), \,\, t_{3}(y)<\eta(y)<t_{4}(y),$$
 with
 $$t_{1}(y)=u_{\lambda}(x_{0})-u_{\lambda}(y), \,\, t_{2}(y)=u(x_{0})-u_{\lambda}(y),$$
 $$t_{3}(y)=u_{\lambda}(x_{0})-u(y), \,\, t_{4}(y)=u(x_{0})-u(y).$$

 To estimate $I_{1}$, we have known that $G$ is  non-decreasing, therefore
 $$u_{\lambda}(x_{0})-u_{\lambda}(y)-u(x_{0})-u(y)=\omega_{\lambda}(x_{0})-\omega_{\lambda}(y)\leq 0,$$
 $$G(u_{\lambda}(x_{0})-u_{\lambda}(y))-G(u(x_{0})-u(y))\leq 0.$$
 Since
 $$\frac{1}{|x_{0}-y|^{n+sp}} > \frac{1}{|x_{0}-y^{\lambda}|^{n+sp}},$$
 we obtain that $I_{1}\leq 0$.

 For $I_{2}$, we apply Lemma \ref{xilbd}, it follows that
  \begin{equation}
  \begin{aligned}
  I_{2}&=\omega_{\lambda}(x_{0})\int_{\Sigma_{\lambda}}\frac{G^{'}(\xi(y))+G^{'}(\eta(y))}{|x_{0}-y^{\lambda}|^{n+sp}}dy\\
       &\leq \omega_{\lambda}(x_{0})\int_{\Sigma_{\lambda}} \frac{2c_{1}|u(x_{0})-u(y)|^{p-2}}{|x_{0}-y^{\lambda}|^{n+sp}}dy\\
       &\leq \omega_{\lambda}(x_{0})\int_{\Sigma_{\lambda} \backslash (\Omega_{\lambda})} \frac{2c_{1}|u(x_{0})|^{p-2}}{|x_{0}-y^{\lambda}|^{n+sp}}dy\\
       &\leq \omega_{\lambda}(x_{0}) \frac{c_{2}|u(x_{0})|^{p-2}}{\delta^{sp}},
  \end{aligned}
  \end{equation}
  where $\delta$ is the width of the region $\Omega_{\lambda}$ in the $x_{1}$-direction, $c_{2}$ is a positive constant. Hence, taking $\lambda \to -1$,  the region $\Omega_{\lambda}$ is a bounded narrow region, the reflection region is also a bounded narrow region, i.e., $\delta \to 0$,  which implies that
 $$I_{2}\leq \omega_{\lambda}(x_{0}) \frac{c_{2}|u(x_{0})|^{p-2}}{\delta^{sp}}\to -\infty.$$

Meanwhile, $a(x)$ is bounded from below and the non-decrease property of  the function $G$, Hence when $\lambda \to -1$, according to \eqref{Lul0} it implies that
\begin{equation}\nonumber
\begin{aligned}
 & \scrL _{p} u_{\lambda}(x_{0})-\scrL _{p} u(x_{0})\\
 =&\ C_{n,s,p}P.V.(I_{1}+I_{2})+
 a(x_{0}) \left(G_{\lambda}(u(x_{0}))-G(u(x_{0}))\right)\\
 \to &\  -\infty.
 \end{aligned}
 \end{equation}

However, since $u$ is bounded in $\Omega$ and f is Lipschitz continuous, it follows that
\begin{equation}\nonumber
\begin{aligned}
&\scrL _{p} u_{\lambda}(x_{0})-\scrL _{p} u(x_{0}) \\
= &\ f(u_{\lambda}(x_{0}))-f(u(x_{0}))\\
 =&\ \omega_{\lambda}(x_{0})\frac{f(u_{\lambda}(x_{0}))-f(u(x_{0}))}{u_{\lambda}(x_{0})-u(x_{0})}\\
\geq &\ -C(L)>-\infty,
\end{aligned}
\end{equation}
where $C(L)$ is a positive constant depending on the Lipschitz constant $L$ of $f$. That makes a contradictory. Hence, when $\lambda \to -1$, the claim \eqref{omegald} holds, i.e.,
$$\omega_{\lambda} (x)\geq 0,\quad  x\in \Sigma_{\lambda}.$$

$Step2$. We move the plane $T_{\lambda}$ to the right continuously  until \eqref{omegald} does not hold.
Define
\begin{equation}\label{lbd0}
\lambda_{0}:=\sup\{\lambda \leq 0\,|\,\omega_{\mu} (x)\geq 0,x\in \Sigma_{\mu},\mu \leq \lambda \},
\end{equation}
we claim that
\begin{equation}\label{lbd00}
\lambda_{0}=0.
\end{equation}

Assume on the contrary that $\lambda_{0}<0$. Then there exists $y\in \Sigma_{\lambda_{0}}\backslash \Omega_{\lambda_{0}}$ such that $\omega_{\lambda_{0}} (y)> 0$. From \eqref{lbd0}, we have
\begin{equation}\label{omelam0not0}
\omega_{\lambda_{0}} (x)\geq 0\,\,\,\mbox{but}\,\, \not\equiv 0,  \,\,\,\mbox{in} \,\,\Sigma_{\lambda_{0}}.
\end{equation}
In the following, we will show that when the plane $T_{\lambda_{0}}$ moves right a little, i.e., $\lambda_{0} \to \lambda_{0}+\varepsilon \leq 0$, where $\varepsilon>0$  small enough, the inequality
$$\omega_{\lambda_{0}+\varepsilon} (x)\geq 0, \,\,\,\mbox{for} \,\,x\in\Sigma_{\lambda_{0}+\varepsilon}$$
still holds, which contradicts with the definition of $\lambda_{0}$.  Since for $x\in \Sigma_{\lambda_{0}+\varepsilon}\backslash \Omega_{\lambda_{0}+\varepsilon}$, $u_{\lambda_{0}+\varepsilon} (x)\geq 0$ and $u(x)=0$, $\omega_{\lambda_{0}+\varepsilon} (x)\geq 0.$ Then this process will be divided into the following two parts,
\begin{equation}
 \begin{aligned}
 &(1)\,\,\omega_{\lambda_{0}+\varepsilon}(x)\geq 0,  \,\,\, \mbox{in}\,\,\Omega_{\lambda_{0}}.\\
 &(2)\,\, \omega_{\lambda_{0}+\varepsilon}(x)\geq 0,  \,\,\, \mbox{in}\,\,\Omega_{\lambda_{0}+\varepsilon}\backslash \Omega_{\lambda_{0}}.
 \end{aligned}
 \end{equation}

(1)\,\,$\omega_{\lambda_{0}+\varepsilon}(x)\geq 0,  \,\,\, \mbox{in}\,\,\Omega_{\lambda_{0}}.$

From the definition of $\lambda_{0}$, it implies that
$$\omega_{\lambda_{0}}(x)\geq 0, \,\,\, \mbox{in}\,\,\Omega_{\lambda_{0}}.$$
in the following, we will prove that
\begin{equation}\label{omelam0}
\omega_{\lambda_{0}}(x)> 0, \,\,\, \mbox{in}\,\,\Omega_{\lambda_{0}}.
\end{equation}

Indeed, if there is a point $ x_{0} \in \Omega_{\lambda_{0}}$ such that
\begin{equation}\label{omegalam0x00}
\omega_{\lambda_{0}}(x_{0})=0= \min_{\Omega_{\lambda_{0}}}\omega_{\lambda_{0}}(x),
\end{equation}
then, we have that
 \begin{equation}\label{i1i2}
 \begin{aligned}
 &\scrL _{p} u_{\lambda_{0}}(x_{0})-\scrL _{p} u(x_{0})\\
 =&\ (-\Delta)_{p}^{s}u_{\lambda_{0}}(x_{0})-(-\Delta)_{p}^{s}u(x_{0})+a(x_{0})\left(|u_{\lambda_{0}}|^{p-2}u_{\lambda_{0}}(x_{0})-|u|^{p-2}u(x_{0})\right)\\
   =&\ C_{n,s,p}P.V.\int_{\R^{n}}\frac{G(u_{\lambda_{0}}(x_{0})-u_{\lambda_{0}}(y))-G(u(x_{0})-u(y))}{|x_{0}-y|^{n+sp}}dy+0\\
   =&\ C_{n,s,p}P.V.\int_{\Sigma_{\lambda_{0}}}\left(G(u_{\lambda_{0}}(x_{0})-u_{\lambda_{0}}(y))-G(u(x_{0})-u(y))\right)\left(\frac{1}{|x_{0}-y|^{n+sp}}-\frac{1}{|x_{0}-y^{\lambda_{0}}|^{n+sp}}\right)dy\\
    &\ +C_{n,s,p}P.V.\int_{\Sigma_{\lambda_{0}}}\frac{G(u_{\lambda_{0}}(x_{0})-u_{\lambda_{0}}(y))-G(u(x_{0})-u(y))+G(u_{\lambda_{0}}(x_{0})-u(y))-G(u(x_{0})-u_{\lambda_{0}}(y))}{|x_{0}-y^{\lambda_{0}}|^{n+sp}}dy\\
   =&\ C_{n,s,p}(I_{1}+I_{2}).
 \end{aligned}
 \end{equation}
According to the fact that \eqref{omelam0not0}, we apply the same argument in $Step1$. For $y\in \Sigma_{\lambda_{0}}$, it follows that
$$\left(u_{\lambda_{0}}(x_{0})-u_{\lambda_{0}}(y)\right)-\left(u(x_{0})-u(y)\right)=\omega_{\lambda_{0}}(x_{0})-\omega_{\lambda_{0}}(y)\leq 0, \,\,\,\mbox{but}\,\,\not\equiv 0.$$
By the monotonicity of $G$, it follows that
$$G(u_{\lambda_{0}}(x_{0})-u_{\lambda_{0}}(y))-G(u(x_{0})-u(y))\leq 0, \,\,\,\mbox{but}\,\,\not\equiv 0,$$
which implies that $I_{1}<0$.

For $I_{2}$, we apply \eqref{omegalam0x00} and take the same process in $Step1$, it follows that
$$I_{2}=\omega_{\lambda_{0}}(x_{0})\int_{\Sigma_{\lambda_{0}}}\frac{G^{'}(\xi(y))+G^{'}(\eta(y))}{|x_{0}-y^{\lambda_{0}}|^{n+sp}}dy=0,$$
Then, we have
\begin{equation}\label{dyb}
 \scrL _{p} u_{\lambda_{0}}(x_{0})-\scrL _{p} u(x_{0}) =C_{n,s,p}(I_{1}+I_{2})<0.
 \end{equation}

However, $u_{\lambda_{0}}(x_{0})=u(x_{0}),$ it implies that
$$\scrL _{p} u_{\lambda_{0}}(x_{0})-\scrL _{p} u(x_{0})=f(u_{\lambda_{0}}(x_{0}))-f(u(x_{0}))=0,$$
which contradicts \eqref{dyb}. Therefore, we obtain \eqref{omelam0}, i.e.,
$$\omega_{\lambda_{0}}(x)> 0, \,\,\, \mbox{in} \,\,\Omega_{\lambda_{0}}.$$
Since $\omega$ is continuous w.r.t. $\lambda$, then there exists $\varepsilon>0$ small enough such that
\begin{equation}\label{olv}
\omega_{\lambda_{0}+\varepsilon}(x)\geq 0, \,\,\,  \mbox{in}\,\,\Omega_{\lambda_{0}}.
\end{equation}
The first part result is obtained.

(2)\,\, $\omega_{\lambda_{0}+\varepsilon}(x)\geq 0,  \,\,\, \mbox{for}\,\,x\in\Omega_{\lambda_{0}+\varepsilon}\backslash \Omega_{\lambda_{0}}.$

 Indeed, since $\varepsilon$ is small enough, the domain $\Omega_{\lambda_{0}+\varepsilon} \backslash  \Omega_{\lambda_{0}}$ is a bounded narrow region, we take the same argument as in $Step1$, and obtain that
 \begin{equation}\label{olv2}
\omega_{\lambda_{0}+\varepsilon}(x)\geq 0, \,\,\, \textup{in}\,\,\Omega_{\lambda_{0}+\varepsilon}\backslash \Omega_{\lambda_{0}}.
\end{equation}
We get second part result.

Combined \eqref{olv} with \eqref{olv2}, we have
$$\omega_{\lambda_{0}+\varepsilon} (x)\geq 0, \,\,\,~\mbox{in}\,\,\Omega_{\lambda_{0}+\varepsilon}$$
which contradicts the definition of $\lambda_{0}$, it follows that \eqref{lbd00} holds, i.e., $\lambda=0$. Hence,
$$\omega_{0}(x)\geq 0, \,\,\,x\in \Omega_{0}.$$

If the starting point of moving plane is on right side, then $T_\lambda$ moves toward left continuously from $\lambda=1$. Follow the similar argument, we can obtain
$$\omega_{0}(x)\leq 0, \,\,\,x\in \Omega_{0}.$$
Therefore, we have
$$\omega_{0}(x)\equiv 0, \,\,\,x\in B_{1}(0).$$

Since the $x_{1}-$direction is chosen arbitrarily, we have actually shown that the solution $u(x)$ of equation \eqref{SUM} is radially symmetric about the origin in $B_{1}(0).$
\end{proof}

\section{The proof of Theorem \ref{Them2}}
 For the unbounded domain case, we will apply the idea of the sliding method and construct the auxiliary functions, the domains of which change along with the maximizing sequence, to prove Theorem \ref{Them2}. We firstly explore that the positive solution converges to the upper bound $\mu$ as the sequence $\{x_k\}$ that we construct tends to infinity, to show Proposition \ref{thm2-2}. Then, we derive the strict monotonicity of the solution of the equation in \eqref{SUM}.

\begin{Prop} \label{thm2-2}
Let $u \in C_{loc}^{1,1}(\Omega) \cap \calL_{sp}$ be a bounded and positive solution of the equation \eqref{SUM}. Assume that $a(x)$ and $f(u)$ satisfy the conditions $(A_{2})$ and $(F_{1})$-$(F_{3})$ in Theorem \ref{Them2} relatively. Then $u(x) \to \mu $ as $x_n\to \infty$.
\end{Prop}

 Firstly, we will explore that there exists a positive lower bound of the solution $u$ at infinity, which is a necessary condition for Proposition \ref{thm2-2}.
\begin{Lem} \label{lem2-3}
Let $a$ and $f$ satisfy conditions in Proposition \ref{thm2-2}, then there exist  $\widehat{c}$, $R_{0}>0$ such that for $R_{0}\geq R_{1}$,
$$u(x)>\widehat{c}, \quad\,~\mbox{as}~ \,\, dist(x,\partial \Omega)>R_{0},$$
where $R_{1}$ is the same as in condition $(A_{2})$ in Theorem \ref{Them2}.
\end{Lem}

\begin{proof}
Let's choose an auxiliary function $\phi\in C_{0}^{\infty}(\R^{n})$ as follows,
\begin{equation}\label{phix}
\phi(x)=
\begin{cases}
ce^{\frac{1}{|x|^{2}-1}}, & |x|<1,\\
0, & |x|\geq 1,
\end{cases}
\end{equation}
where $c>0$ with $\phi(0)=1$. For $z\in\Omega$, take $R>0$ such that $B_{R}(z)\subset\Omega$. Define
$$\phi_{R}(x):=\phi(\frac{x-z}{R}),\,\,\,\phi_{\varepsilon,R}(x):=\varepsilon\phi(\frac{x-z}{R}).$$
It's obvious that $\phi_{\varepsilon,R}(z)=\max_{x\in B_{R}(z)}\phi_{\varepsilon,R}(x)=\varepsilon$. By calculating we can imply that
\begin{equation} \label{phCvare}
\left\{
\begin{aligned}
&(-\Delta)_{p}^{s}\phi_{\varepsilon,R}(x)\leq \frac{C}{R^{sp}}\varepsilon^{p-1},\,\,\mbox{in}\,\,B_{R}(z),\\
&\phi_{\varepsilon,R}(x)=0,\,\,\mbox{on}\,\,B^{c}_{R}(z),
\end{aligned}
\right.
\end{equation}
where $C$ is a positive constant.

For $y_{0}\in \Omega$ with  $dist(y_{0},\partial \Omega)> R_{0}>R_{1}$,  we choose $\varepsilon_{0}>0$ small enough such that
$$\varepsilon_{0}<\inf_{x\in B_{R_{0}}(y_{0})}u(x).$$
Let $0<\varepsilon_{1}<\min\{\varepsilon_{0}, t_{0}\}$, we have that
\begin{equation}\label{uvpR0}
u(x)>\varepsilon_{1}\phi_{R_{0}}(x)=\phi_{\varepsilon_{1},R_{0}}(x),\,\,\mbox{for}\,\,x\in B_{R_{0}}(y_{0}).
\end{equation}
We set $t\in[0,1]$, $y\in\Omega$ with $dist(y, \partial\Omega)>R_{0}$, $y_{t}=ty+(1-t)y_{0}$  and $$\omega_{t}(x):=u(x)-\phi_{\varepsilon_{1},R_{0}}(x),\,\,\,\mbox{for}\,\,x\in \overline{B_{R_{0}}(y_{t})}.$$
According to \eqref{uvpR0}, we have
$$\omega_{0}(x)=u(x)-\phi_{\varepsilon_{1},R_{0}}(x)>0,\,\,\,\mbox{for}\,\,x\in \overline{B_{R_{0}}(y_{0})},$$
and
\begin{equation}\label{partialB}
\omega_{t}(x)>0,\,\,\,\mbox{for}\,\, x \in \partial B_{R_{0}}(y_{t}).
\end{equation}

Now, we claim that
\begin{equation} \label{wtR}
\omega_{t}(x)>0, \,\,\, \mbox{for}\,\, x\in B_{R_{0}}(y_{t}).
\end{equation}
in other words, i.e.
\begin{equation} \label{tiomt}
\tilde{\omega}(t):=\min_{x\in B_{R_{0}}(y_{t})}(u(x)-\phi_{\varepsilon_{1},R_{0}}(x))> 0,\,\,\mbox{for}\,\, x\in B_{R_{0}}(y_{t}).
\end{equation}
In the following, we prove the claim \eqref{wtR} by contradiction. Since $\tilde{\omega}$ is continuous about $t$ and $\tilde{\omega}(0)>0$, assume on the contrary, if there exists $t_{0}$ such that the graph of $\phi_{\varepsilon_{1},R_{0}}$ touches the graph of $u$ at some point $x_{t_{0}}\in B_{R_{0}}(y_{t_{0}})$, i.e., $\tilde{\omega}(t_{0})=0,$ $\omega_{t}(x_{t_{0}})=0$, then we have that
\begin{equation}\label{tvap1R}
 \begin{aligned}
 &\scrL_{p}(u)(x_{t_{0}})-\scrL_{p}(\phi)(x_{t_{0}})\\
=&\ (-\Delta)_{p}^{s}u(x_{t_{0}})+a(x_{t_{0}})|u|^{p-2}u(x_{t_{0}})-  \left((-\Delta)_{p}^{s}\phi_{\varepsilon_{1},R_{0}}(x_{t_{0}})+a(x_{t_{0}})|\phi_{\varepsilon_{1,R_{0}}}|^{p-2}\phi_{\varepsilon_{1,R_{0}}}(x_{t_{0}})\right)\\
=&\ C_{n,s,p}P.V.\int_{\R^{n}}\frac{G\left(u(x_{t_{0}})-u(y)\right)-G\left(\phi_{\varepsilon_{1},R_{0}}(x_{t_{0}})-\phi_{\varepsilon_{1},R_{0}}(y)\right)}{|x_{t_{0}}-y|^{n+sp}}dy\\
=&\ C_{n,s,p}P.V.\int_{B_{R_{0}}(y_{t_{0}})}\frac{G\left(u(x_{t_{0}})-u(y)\right)-G\left(\phi_{\varepsilon_{1},R_{0}}(x_{t_{0}})-\phi_{\varepsilon_{1},R_{0}}(y)\right)}{|x_{t_{0}}-y|^{n+sp}}dy\\
&\ +C_{n,s,p}P.V.\int_{\R^{n}\backslash B_{R_{0}}(y_{t_{0}})} \frac{G\left(u(x_{t_{0}})-u(y)\right)-G\left(\phi_{\varepsilon_{1},R_{0}}(x_{t_{0}})\right)}{|x_{t_{0}}-y|^{n+sp}}dy\\
=&\ C_{n,s,p}P.V.(I_{1}+I_{2}).
\end{aligned}
\end{equation}

 For $I_{1}$, since $y\in B_{R_{0}}(y_{t_{0}})$ and the monotonicity of $G$, we have
$$\left(u(x_{t_{0}})-u(y)\right)-\left(\phi_{\varepsilon_{1},R_{0}}(x_{t_{0}})-\phi_{\varepsilon_{1},R_{0}}(y)\right)=\omega_{t_{0}}(x_{t_{0}})-\omega_{t_{0}}(y)\leq 0,$$
which implies that
$$G\left(u(x_{t_{0}})-u(y)\right)-G\left(\phi_{\varepsilon_{1},R_{0}}(x_{t_{0}})-\phi_{\varepsilon_{1},R_{0}}(y)\right)\leq 0.$$
Therefore, we obtain that
$$I_{1}\leq 0.$$

For $I_{2}$, due to $y\in \R^{n}\backslash B_{R_{0}}(y_{t_{0}})$, we can deduce
$$u(x_{t_{0}})-u(y)-\phi_{\varepsilon_{1},R_{0}}(x_{t_{0}})=-u(y)\leq 0\,\,\,\,\mbox{but}\,\,\not\equiv 0.$$
It yields that
$$G\left(u(x_{t_{0}})-u(y)\right)-G\left(\phi_{\varepsilon_{1},R_{0}}(x_{t_{0}})\right)\leq 0\,\,\,\,\,\mbox{but}\,\,\not\equiv 0.$$
Therefore we have
$$I_{2}<0.$$
Hence, from $\eqref{tvap1R}$ we have
\begin{equation}\label{utphva1R}
\scrL_{p}(u)(x_{t_{0}})-\scrL_{p}(\phi)(x_{t_{0}})<0.
\end{equation}

On the other hand, since the estimation \eqref{phCvare} for $\phi$ and the choice of $\varepsilon_{1}$, we can deduce that
\begin{equation}\label{R0spc0}
 \begin{aligned}
&\scrL_{p}(u)(x_{t_{0}})-\scrL_{p}(\phi_{\varepsilon_{1},R_{0}})(x_{t_{0}})\\
=&\ (-\Delta)_{p}^{s}u(x_{t_{0}})+a(x_{t_{0}})|u|^{p-2}u(x_{t_{0}})-  \left((-\Delta)_{p}^{s}\phi_{\varepsilon_{1},R_{0}}(x_{t_{0}})+a(x_{t_{0}})|\phi_{\varepsilon_{1},R_{0}}|^{p-2}\phi_{\varepsilon_{1},R_{0}}(x_{t_{0}})\right)\\
\geq &\ f(u(x_{t_{0}}))-\frac{C\varepsilon_{1}^{p-1}}{R_{0}^{sp}}-a(x_{t_{0}})|\phi_{\varepsilon_{1},R_{0}}|^{p-2}\phi_{\varepsilon_{1},R_{0}}(x_{t_{0}})\\
\geq &\ \delta_{0} \phi_{\varepsilon_{1},R_{0}}(x_{t_{0}})-\frac{C\varepsilon_{1}^{p-1}}{R_{0}^{sp}}-a(x_{t_{0}})|\phi_{\varepsilon_{1},R_{0}}|^{p-2}\phi_{\varepsilon_{1},R_{0}}(x_{t_{0}}).
\end{aligned}
\end{equation}
 When $R_{0}\rightarrow\infty$, we have $\frac{C\varepsilon_{1}^{p-1}}{R_{0}^{sp}}\rightarrow0$ and $dist(x_{t_{0}},\partial\Omega)\rightarrow\infty$.

 Thus, by the condition $(A_{2})$ in Theorem \ref{Them2}, we obtain that from $\eqref{R0spc0}$
$$\scrL_{p}(u)(x_{t_{0}})-\scrL_{p}(\phi)(x_{t_{0}})\geq \delta_{0} \phi_{\varepsilon_{1},R_{0}}(x_{t_{0}})>0,$$
which conflicts \eqref{utphva1R}. Therefore, we finish the proof of the claim \eqref{wtR}. When  $t=1$ in \eqref{tiomt}, we obtain
$$u(x)>\varepsilon_{1}\phi(\frac{x-y}{R_{0}}), \,\,\mbox{for}\,\,x\in \overline{B_{R_{0}}(y)}.$$
In particular, when $x=y$, it implies that $\varepsilon_{1}\phi(\frac{x-y}{R_{0}})=\varepsilon_{1}$. Set a positive constant $\widehat{c}\leq \varepsilon_{1}$, it yields that
$$u(x)>\widehat{c},\,\,\,\mbox{for}\,\, x\in\Omega\,\,\mbox{with} \,\,dist(x,\partial\Omega)>R_{0}.$$
This completes the proof of the lemma.
\end{proof}

%\begin{Rem}
%Indeed, we only need the condition $(F_{1})$ and $(F_{2})$ for $f$ in the above lemma, and $a(x)$ admits a weaker condition such as $\varepsilon'\in(0,\varepsilon_{1})$, $a(x)<\frac{c\delta_{0}}{|\varepsilon'|^{p-2}}$
% when $dist(x,\partial \Omega)>R_{0}$. The existence of  lower bound of the solution still holds.
%\end{Rem}

Now, we give the proof of Proposition \ref{thm2-2}.
\begin{proof}
Firstly, we choose the same auxiliary function $\phi(x)$ as \eqref{phix} in Lemma \ref{lem2-3}. Here we set $\phi(0)=\mu$, and $\phi_{R}(x):=\phi(\frac{x-x_{R}}{R})$, where $x_{R}$ is chosen by $dist(x_{R},\partial\Omega)>2R>R_{0}$, with $B_{2R}(x_{R})\subset \Omega$. It's obvious that $\phi_{R}(x_{R})=\max_{x\in B_{R}(x_{R})}\phi_{R}(x)=\phi(0)=\mu$. By the same calculation as \eqref{phCvare}, we have that
\begin{equation}\label{410}
\left\{
\begin{aligned}
&(-\Delta)_{p}^{s}\phi_{R}(x)\leq\frac{C}{R^{sp}}\mu^{p-1},\,\,\,\mbox{in}\,B_{R}(x_{R}),\\
&\phi_{R}(x)=0,\,\,\,\mbox{on}\, B^{c}_{R}(x_{R}).
\end{aligned}
\right.
\end{equation}

On the one hand, from the equation \eqref{SUM} and \eqref{410}, it yields that
\begin{equation}\label{fRa}
\begin{aligned}
&\scrL_{p}(u)(x)-\scrL_{p}(\phi_{R})(x)\\
=&\ (-\Delta)_{p}^{s}u(x)+a(x)|u|^{p-2}u(x)-\left((-\Delta)_{p}^{s}\phi_{R}(x)+a(x)|\phi_{R}|^{p-2}\phi_{R}(x)\right)\\
\geq &\ f(u(x))-\frac{C}{R^{sp}}\mu^{p-1}-a(x)|\phi_{R}|^{p-2}\phi_{R}(x).
\end{aligned}
\end{equation}

On the other hand, let $\omega_{R}(x):= u(x)-\phi_{R}(x)$, since $\phi_{R}(x_{R})=\mu$ and $u(x)<\mu$, we can infer that there exists $\overline{x}_{R}$ such that
$$\omega_{R}(\overline{x}_{R})=\min_{x\in B_{R}(x_{R})} \omega_{R}(x)<0.$$
It follows that
$$u(\overline{x}_{R})-\phi_{R}(\overline{x}_{R})\leq u(x_{R})-\phi_{R}(x_{R}),$$
which implies
\begin{equation}\label{uop}
u(\overline{x}_{R})\leq u(x_{R})-\left(\phi_{R}(x_{R})-\phi_{R}(\overline{x}_{R})\right)\leq u(x_{R}).
\end{equation}
 We have that
\begin{equation}\nonumber
\begin{aligned}
&\scrL_{p}(u)(\overline{x}_{R})-\scrL_{p}(\phi_{R})(\overline{x}_{R})\\
=&\ (-\Delta)_{p}^{s}u(\overline{x}_{R})+a(\overline{x}_{R})|u|^{p-2}u(\overline{x}_{R})-
\left((-\Delta)_{p}^{s}\phi_{R}(\overline{x}_{R})+a(\overline{x}_{R})|\phi_{R}|^{p-2}\phi_{R}(\overline{x}_{R})\right)\\
=&\ C_{n,s,p}P.V.\int_{\R^{n}}\frac{G\left(u(\overline{x}_{R})-u(y)\right)-G\left(\phi_{R}(\overline{x}_{R})-\phi_{R}(y)\right)}{|\overline{x}_{R}-y|^{n+sp}}dy\\
&\ +a(\overline{x}_{R})|u|^{p-2}u(\overline{x}_{R})-a(\overline{x}_{R})|\phi_{R}|^{p-2}\phi_{R}(\overline{x}_{R})\\
= &\ C_{n,s,p}P.V.\int_{B_{R}(x_{R})}\frac{G\left(u(\overline{x}_{R})-u(y)\right)-G\left(\phi_{R}(\overline{x}_{R})-\phi_{R}(y)\right)}{|\overline{x}_{R}-y|^{n+sp}}dy\\
&\ +C_{n,s,p}P.V.\int_{\R^{n}\backslash B_{R}(x_{R})} \frac{G\left(u(\overline{x}_{R})-u(y)\right)-G\left(\phi_{R}(\overline{x}_{R})\right)}{|\overline{x}_{R}-y|^{n+sp}}dy\\ &\ +a(\overline{x}_{R})|u|^{p-2}u(\overline{x}_{R})-a(\overline{x}_{R})|\phi_{R}|^{p-2}\phi_{R}(\overline{x}_{R})\\
=&\ C_{n,s,p}P.V.(I_{1}+I_{2})+I_{3}.
\end{aligned}
\end{equation}
For $I_{1}$, due to $y\in B_{R}(x_{R})$ and the monotonicity of $G$, we have that
$$\left(u(\overline{x}_{R})-u(y)\right)-\left(\phi_{R}(\overline{x}_{R})-\phi_{R}(y)\right)=\omega_{R}(\overline{x}_{R})-\omega_{R}(y)\leq 0.$$
It follows that
$$G\left(u(\overline{x}_{R})-u(y)\right)-G\left(\phi_{R}(\overline{x}_{R})-\phi_{R}(y)\right)\leq 0.$$
Hence, we get that $$I_{1}\leq 0.$$

For $I_{2}$, $y\in \R^{n}\backslash B_{R}(x_{R})$, we have
$$u(\overline{x}_{R})-u(y)-\phi_{R}(\overline{x}_{R})=\omega_{R}(\overline{x}_{R})-u(y)\leq 0,$$
which implies that
$$G(u(\overline{x}_{R})-u(y))-G(\phi_{R}(\overline{x}_{R}))\leq 0.$$
So, we have that $$I_{2}\leq 0.$$
Therefore, we obtain that
\begin{equation}\label{fRa2}
\scrL_{p}(u)(\overline{x}_{R})-\scrL_{p}(\phi_{R})(\overline{x}_{R})\leq I_{3}.
\end{equation}

  Combing \eqref{fRa} with \eqref{fRa2} , we have that
 $$ f(u(x))-\frac{C}{R^{sp}}\mu^{p-1}\leq a(\overline{x}_{R})|u|^{p-2}u(\overline{x}_{R}),$$
 from Lemma \ref{lem2-3} and the condition $(A_{2})$, i.e. $|a(x)|\thicksim\frac{1}{|x|^{\alpha}}$, we can imply that
$$0<f(u(\overline{x}_{R}))\leq \left(\frac{1}{|{x}_{R}|^{\alpha}}+\frac{C}{R^{sp}}\right)\mu^{p-1}.$$
Hence, we can obtain that
\begin{equation}
f(u(\overline{x}_{R}))\to 0,\,\,\mbox{as}\,\,R\to \infty .\label{fuR}
\end{equation}

Now, we claim that
\begin{equation}\label{uoverxRmu}
u(\overline{x}_{R})\to \mu, \,\,\,\mbox{as}\,\,R \to \infty.
\end{equation}
In fact, by Lemma \ref{lem2-3} we have $u(\overline{x}_{R})>\widehat{c}>0$ when the point $\overline{x}_{R}$ is far away from the boundary. Suppose on the contrary, $u(\overline{x}_{R})\nrightarrow \mu$, then there exist two cases as follows.

$Case1$. For $u(\overline{x}_{R})\in(0,t_{0}]$. Since the condition $(F_{2})$, we have
$$f(u(\overline{x}_{R}))\geq \delta_{0}u(\overline{x}_{R})\geq \delta_{0}\varepsilon_{1}>0,$$
which contradicts \eqref{fuR}. So, the $Case1$ can't hold.

$Case2$. For $u(\overline{x}_{R})\in[t_{0},t_{1}]$. From the condition $(F_{1})$, i.e., $f(u(x))>0$ as $u(x)\in (0,\mu)$ and $f$ is continuous, it deduces that
$$\min_{u(x)\in[t_{0},t_{1}]}f(u(x))=c_{t}>0,$$
where $c_{t}$ is a constant. That conflicts with \eqref{fuR}, the $Case2$ is not valid.

Thus, we derive from the above two cases and the fact \eqref{fuR} that $u(\overline{x}_{R})$ must fall in open interval $(t_{1},\mu)$, in which $f$ is non-increasing in the condition$(F_{3})$. It infers that \eqref{uoverxRmu} must be valid, the claim is obtained.

From \eqref{uop}, we have
$$\mu\leftarrow u(\overline{x}_{R})\leq u(x_{R})<\mu,\,\,\,\mbox{as}\,\, R\rightarrow\infty.$$
Therefore, $u(x_{R})\rightarrow \mu$ as $ R\rightarrow \infty $. The proof of Proposition \ref{thm2-2} is completed.
\end{proof}

%\begin{Rem}
%In fact, if $a(x)$ is non-positive near infinity, the condition $(A')$ is not necessary in Lemma \ref{lem2-3} and Proposition \ref{thm2-2}.
%
%Moreover, from Proposition \ref{thm2-2} we can not difficult to find that  $M_{1}$ is large enough to satisfy $\{x\,|\,0<dist(x,\partial\Omega)<M_{1}\}$ be the domain of $u\in(t_{1},\mu)$ in the assumption $(A)$.
%\end{Rem}

Proposition \ref{thm2-2} indicates that there exists a maximizing sequence of the solution of the equation \eqref{SUM} near infinity, which is the most important property of the sliding process.

In the following, we prove Theorem \ref{Them2} completely.
\begin{proof}
For $\tau\geq 0$, denote
$$u_{\tau}(x):=u(x+\tau e_{n}),\,\,\, \omega_{\tau}(x):=u(x)-u_{\tau}(x),$$
where $e_{n}=(0,0,\cdots,1)$.  We will carry out the proof in two steps.

$Step 1$. For $\tau>0$ sufficiently large, We will show that
\begin{equation}
\omega_{\tau}(x)\leq 0,\,\,\,\mbox{for}\,\,x\in \R^{n}.  \label{ota}
\end{equation}
This provides the starting point for the sliding method. For $h>0$, define
$$\Omega_{h}:=\{x\in \R^{n}|\,dist(x,\partial\Omega)<h\}.$$
By Proposition \ref{thm2-2}, there exists a constant $M_{0}>0$ large enough such that for $\tau>M_{0}$, we have
\begin{equation}
u_{\tau}(x)\in (t_{1},\mu), \,\,\,\mbox{for}\,\,x\in\Omega. \label{ut1}
\end{equation}

Assume on the contrary that \eqref{ota} is violated, i.e., there exists a constant $A$ such that
$$\sup_{x\in \R^{n}}\omega_{\tau}(x)=A>0.$$
Then there exists sequences $\{x^{k}\}\subseteq \R^{n}$ and $\{\gamma_{k}\}\subseteq(0,1)$ with $\gamma_{k}\rightarrow1$ as $k\rightarrow\infty$ such that
\begin{equation} \label{omtaxkga}
\begin{aligned}
&\omega_{\tau}(x^{k})\geq \gamma_{k}A,\\
&\omega_{\tau}(x^{k})\rightarrow A, \,\,\, \mbox{as} \,\,k \rightarrow \infty.
\end{aligned}
\end{equation}
Since $\omega_{\tau}(x)\leq 0$, for any $x\in \R^{n}\setminus\Omega$, then we have $\{x^{k}\}\nsubseteq \R^{n} \setminus \Omega$. Meanwhile, by Proposition \ref{thm2-2}, $\omega_{\tau}(x)\rightarrow 0\,\,\mbox{as} \,\, dist(x,\partial\Omega)\rightarrow \infty,$ it infers that there exists a constant $M$ with $M>R_{0}>0$ such that
$$\{x^{k}\} \subseteq \Omega_{M}.$$

Let
\begin{equation}
\Phi(x):=
\begin{cases}
ce^{\frac{1}{|x|^{2}-4}} &|x|<2\\
0 &|x|\geq 2
\end{cases}
\end{equation}
where $c>0$ such that $\Phi(0)=1$. Define
\begin{equation}\label{Phirk}
\Phi_{r_{k}}(x):=\Phi(\frac{x-x^{k}}{r_{k}}), \,\,\,x \in B_{2r_{k}}(x^{k}).
\end{equation}

For any $x\in B_{2r_{k}}(x^{k})\setminus B_{r_{k}}(x^{k})$, we take $\varepsilon_{k}, r_{k}>0$  such that
\begin{equation}\label{otvkPr}
\begin{aligned}
\omega_{\tau}(x^{k})+\varepsilon_{k}\Phi_{r_{k}}(x^{k})&\geq A+\varepsilon_{k}\Phi_{r_{k}}(x^{k}+r_{k}e)\\
&\geq \omega_{\tau}(x)+\varepsilon_{k}\Phi_{r_{k}}(x),
\end{aligned}
\end{equation}
with $\varepsilon_{k}\rightarrow0$, $r_{k}\rightarrow0$ as $k\rightarrow\infty$, and $e$ is any unit vector in $\R^{n}$. It follows that there exists $\overline{x}^{k} \in B_{r_{k}}(x^{k})$ such that
\begin{equation} \label{ovP}
\begin{aligned}
\omega_{\tau}(\overline{x}^{k})+\varepsilon_{k}\Phi_{r_{k}}(\overline{x}^{k}) & =\max_{x\in B_{r_{k}}(x^{k})}\left(\omega_{\tau}(x)+\varepsilon_{k}\Phi_{r_{k}}(x)\right) \\
&\geq \omega_{\tau}(x^{k})+\varepsilon_{k}\Phi_{r_{k}}(x^{k}).
\end{aligned}
\end{equation}
It implies that
\begin{equation}\label{omtaovga}
A\geq\omega_{\tau}(\overline{x}^{k})\geq \omega_{\tau}(x^{k})+\left(\varepsilon_{k}\Phi_{r_{k}}(x^{k})-\varepsilon_{k}\Phi_{r_{k}}(\overline{x}^{k})\right)\geq \omega_{\tau}(x^{k})\geq \gamma_{k}A>0.
\end{equation}
From \eqref{otvkPr} and \eqref{ovP}, we have
\begin{equation} \label{tovxkvP}
\begin{aligned}
&\omega_{\tau}(\overline{x}^{k})+\varepsilon_{k}\Phi_{r_{k}}(\overline{x}^{k})\geq A\geq \omega_{\tau}(x),\,\, \mbox{for}\,\,x\in\R^{n}\\
&\omega_{\tau}(\overline{x}^{k})+\varepsilon_{k}\Phi_{r_{k}}(\overline{x}^{k})=\sup_{x\in\R^{n}}\left(\omega_{\tau}(x)+\varepsilon_{k}\Phi_{r_{k}}(x)\right),\\
&\omega_{\tau}(\overline{x}^{k})\rightarrow A, \,\,\,\,\mbox{as}\,\,k\rightarrow \infty.
\end{aligned}
\end{equation}

We recall the following $L^{p}$-inequality and the result in $Lemma$ $5.2$ in Appendix \cite{WCL}
\begin{equation}\label{pin}
|t_{1}+t_{2}|^{p-2}(t_{1}+t_{2})\leq 2^{p-2}(|t_{1}|^{p-2}t_{1}+|t_{2}|^{p-2}t_{2}),\,\,\mbox{for}\,\,t_{1}+t_{2}>0,\,\,p\geq 2,
\end{equation}
\begin{equation}\label{lem5.2}
\left|(-\Delta)^{s}_{p}(u+\varepsilon_{k}\Phi_{r_{k}})(\overline{x}^{k})-(-\Delta)^{s}_{p}u(\overline{x}^{k})\right|<\varepsilon_{k}C_{1}r_{k}^{-sp}+C_{2}r_{k}^{p(1-s)},
\end{equation}
where $C_{1}$ and $C_{2}$ are independent of $\varepsilon_{k}$.

On the one hand, by \eqref{lem5.2} and the equation \eqref{SUM} we can imply that
\begin{equation} \label{uvtP}
\begin{aligned}
&\scrL_{p}(u+\varepsilon_{k}\Phi_{r_{k}})(\overline{x}^{k})-\scrL_{p}(u_{\tau})(\overline{x}^{k})\\
=&\ (-\Delta)^{s}_{p}(u+\varepsilon_{k}\Phi_{r_{k}})(\overline{x}^{k})-(-\Delta)^{s}_{p}u(\overline{x}^{k})+(-\Delta)^{s}_{p}u(\overline{x}^{k})-(-\Delta)^{s}_{p}u_{\tau}(\overline{x}^{k})\\
&\ +a(\overline{x}^{k})|u+\varepsilon_{k}\Phi_{r_{k}}|^{p-2}(u+\varepsilon_{k}\Phi_{r_{k}})(\overline{x}^{k})-a(\overline{x}^{k})|u_{\tau}|^{p-2}u_{\tau}(\overline{x}^{k})\\
\leq &\ \varepsilon_{k}C_{1}r_{k}^{-sp}+C_{2}r_{k}^{p(1-s)}+ f(u(\overline{x}^{k}))-f(u_{\tau}(\overline{x}^{k}))\\
&\ +a(\overline{x}^{k})|u+\varepsilon_{k}\Phi_{r_{k}}|^{p-2}(u+\varepsilon_{k}\Phi_{r_{k}})(\overline{x}^{k})-a(\overline{x}^{k})|u|^{p-2}u(\overline{x}^{k}).
\end{aligned}
\end{equation}

On the other hand,
\begin{equation} \label{I12a}
\begin{aligned}
&\scrL_{p}(u+\varepsilon_{k}\Phi_{r_{k}})(\overline{x}^{k})-\scrL_{p}(u_{\tau})(\overline{x}^{k})\\
=&\ (-\Delta)^{s}_{p}(u+\varepsilon_{k}\Phi_{r_{k}})(\overline{x}^{k})-(-\Delta)^{s}_{p}u_{\tau}(\overline{x}^{k})+a(\overline{x}^{k})|u+\varepsilon_{k}\Phi_{r_{k}}|^{p-2}(u+\varepsilon_{k}\Phi_{r_{k}})(\overline{x}^{k})\\
&\ -a(\overline{x}^{k})|u_{\tau}|^{p-2}u_{\tau}(\overline{x}^{k})\\
= &\ C_{n,s,p}P.V.\int_{\R^{n}} \frac{G\left(u(\overline{x}^{k})+\varepsilon_{k}\Phi_{r_{k}}(\overline{x}^{k})-u(y)-\varepsilon_{k}\Phi_{r_{k}}(y)\right)-G\left(u_{\tau}(\overline{x}^{k})-u_{\tau}(y)\right)}{|\overline{x}^{k}-y|^{n+sp}}dy\\
&\ +a(\overline{x}^{k})\left(|u+\varepsilon_{k}\Phi_{r_{k}}|^{p-2}(u+\varepsilon_{k}\Phi_{r_{k}})(\overline{x}^{k})-|u_{\tau}|^{p-2}(u_{\tau})(\overline{x}^{k})\right)\\
=&\ C_{n,s,p}P.V.\int_{B_{2r_{k}}(x^{k})}\frac{G\left(u(\overline{x}^{k})+\varepsilon_{k}\Phi_{r_{k}}(\overline{x}^{k})-u(y)-\varepsilon_{k}\Phi_{r_{k}}(y)\right)-G\left(u_{\tau}(\overline{x}^{k})-u_{\tau}(y)\right)}{|\overline{x}^{k}-y|^{n+sp}}dy\\
&\ +C_{n,s,p}P.V.\int_{\R^{n}\backslash B_{2r_{k}}(x^{k})}\frac{G\left(u(\overline{x}^{k})+\varepsilon_{k}\Phi_{r_{k}}(\overline{x}^{k})-u(y)\right)-G\left(u_{\tau}(\overline{x}^{k})-u_{\tau}(y)\right)}{|\overline{x}^{k}-y|^{n+sp}}dy\\
&\ +a(\overline{x}^{k})\left(|u+\varepsilon_{k}\Phi_{r_{k}}|^{p-2}(u+\varepsilon_{k}\Phi_{r_{k}})(\overline{x}^{k})-|u_{\tau}|^{p-2}u_{\tau}(\overline{x}^{k})\right)\\
=:&\  I_{1}+I_{2}+a(\overline{x}^{k})\left(|u+\varepsilon_{k}\Phi_{r_{k}}|^{p-2}(u+\varepsilon_{k}\Phi_{r_{k}})(\overline{x}^{k})-|u_{\tau}|^{p-2}u_{\tau}(\overline{x}^{k})\right).
\end{aligned}
\end{equation}
For $I_{1}$, since $\overline{x}^{k}$ is a maximum point of $\omega_{\tau}+\varepsilon_{k}\Phi_{r_{k}}$, which can imply that
\begin{equation}\nonumber
\begin{aligned}
&\left(u(\overline{x}^{k})+\varepsilon_{k}\Phi_{r_{k}}(\overline{x}^{k})\right)-\left(u(y)+\varepsilon_{k}\Phi_{r_{k}}(y)\right)-\left(u_{\tau}(\overline{x}^{k})-u_{\tau}(y)\right)\\
=&\ \omega_{\tau}(\overline{x}^{k})+\varepsilon_{k}\Phi_{r_{k}}(\overline{x}^{k})-(\omega_{\tau}(y)+\varepsilon_{k}\Phi_{r_{k}}(y))\\
\geq &\  0.
\end{aligned}
\end{equation}
Since $G$ is monotone increasing, it yields that
$$G\left(u(\overline{x}^{k})+\varepsilon_{k}\Phi_{r_{k}}(\overline{x}^{k})-(u(y)+\varepsilon_{k}\Phi_{r_{k}}(y))\right)-G\left(u_{\tau}(\overline{x}^{k})-u_{\tau}(y)\right)\geq 0.$$
Hence, we obtain that
\begin{equation}\label{I1}
I_{1}\geq 0.
\end{equation}

For estimating $I_{2}$, we can imply that by \eqref{pin}
\begin{equation}\nonumber
\begin{aligned}
I_{2}&=\int_{\R^{n}\backslash B_{2r_{k}}(x^{k})}  \frac{G\left(u(\overline{x}^{k})+\varepsilon_{k}\Phi_{r_{k}}(\overline{x}^{k})-u(y)\right)-G\left(u_{\tau}(\overline{x}^{k})-u_{\tau}(y)\right)}{|\overline{x}^{k}-y|^{n+sp}}dy\\
& \geq 2^{2-p}C_{n,s,p}\int_{\R^{n}\backslash B_{2r_{k}}(x^{k})}\frac{G\left(\omega_{\tau}(\overline{x}^{k})+\varepsilon_{k}\Phi_{r_{k}}(\overline{x}^{k})-\omega_{\tau}(y)\right)}{|\overline{x}^{k}-y|^{n+sp}}dy\\
& \geq 2^{2-p}C_{n,s,p}\int_{(\R^{n}\backslash B_{2r_{k}}(x^{k}))\cap(\R^{n}\backslash \Omega)} \frac{G\left(\omega_{\tau}(\overline{x}^{k})+\varepsilon_{k}\Phi_{r_{k}}(\overline{x}^{k})-\omega_{\tau}(y)\right)}{|\overline{x}^{k}-y|^{n+sp}}dy\\
& \geq A^{p-1}2^{2-p}C_{n,s,p} \int_{(\R^{n}\backslash B_{2r_{k}}(x^{k}))\cap(\R^{n}\backslash\Omega)}
\frac{1}{|\overline{x}^{k}-y|^{n+sp}}dy\\
& \geq C_{A,s,p}\int_{(\R^{n}\backslash B_{2r_{k}}(x^{k}))\cap(\R^{n}\backslash \Omega)}\frac{1}{|x^{k}-y|^{n+sp}}dy,\\
\end{aligned}
\end{equation}
where $C_{A,s,p}$ is a positive constant, only depending on $A$, $s$ and $p$. The last inequality holds from the fact
$$|\overline{x}^{k}-y|\leq |\overline{x}^{k}-x^{k}| + |x^{k}-y|\leq \frac{3}{2}|x^{k}-y|.$$
It follows that
\begin{equation} \label{I2}
I_{2}\geq C_{A,s,p}\int_{(\R^{n}\backslash B_{2r_{k}}(x^{k}))\cap(\R^{n}\backslash \Omega)}\frac{1}{|x^{k}-y|^{n+sp}}dy\geq c_{d}>0,
\end{equation}
where $x^{k}\in \Omega_{\tau}$, $c_{d}$ is a positive constant, with depending on the distance of $x_{k}$ from $\partial\Omega$. Thus, we have that
\begin{equation} \label{I1I2crk}
I_{1}+I_{2}\geq c_{d}>0.
\end{equation}

To sum up, we can deduce from \eqref{uvtP}, \eqref{I12a} and \eqref{I1I2crk} that
\begin{equation}\label{cxkvaCrk}
\begin{aligned}
0<c_{d}\leq&\  \varepsilon_{k}C_{1}r_{k}^{-sp}+C_{2}r_{k}^{p(1-s)}+ f(u(\overline{x}^{k}))-f(u_{\tau}(\overline{x}^{k}))\\
&\ -a(\overline{x}^{k})\left(|u|^{p-2}u(\overline{x}^{k})-|u_{\tau}|^{p-2}u_{\tau}(\overline{x}^{k})\right),
\end{aligned}
\end{equation}
Since $\tau$ is large enough, by \eqref{ut1} and \eqref{tovxkvP}, we have known that
$$\mu>u(\overline{x}^{k})>u_{\tau}(\overline{x}^{k})>t_{1}.$$
 Recall the condition $(F_{3})$ that $f(u)$ is  non-increasing for  $u\in(t_{1},\mu)$, we have that
 $$f(u(\overline{x}^{k}))\leq f(u_{\tau}(\overline{x}^{k})).$$
And from the condition $(A_{1})$ i.e., $|a(x)|<\left|\frac{df}{du}\big/\frac{dG}{du}\right|$, we can infer that as $k$ is large enough by applying the Mean Value Theorem
\begin{equation}
\begin{aligned}
&\left(\frac{f(u(\overline{x}^{k}))-f(u_{\tau}(\overline{x}^{k}))}{u(\overline{x}^{k})-u_{\tau}(\overline{x}^{k})} \bigg/
\frac{G(u(\overline{x}^{k}))-G(u_{\tau}(\overline{x}^{k}))}{u(\overline{x}^{k})-u_{\tau}(\overline{x}^{k})}-a(\overline{x}^{k})\right)\\
&\ \times\left(G(u(\overline{x}^{k}))-G(u_{\tau}(\overline{x}^{k}))\right)\\
=&\ \left(\frac{f(u(\overline{x}^{k}))-f(u_{\tau}(\overline{x}^{k}))}{u(\overline{x}^{k})-u_{\tau}(\overline{x}^{k})} \bigg/ G'(\xi(\overline{x}^{k}))-a(\overline{x}^{k}) \right) \left(G(u(\overline{x}^{k}))-G(u_{\tau}(\overline{x}^{k}))\right)\\
\leq &\ \left(\frac{f(u(\overline{x}^{k}))-f(u_{\tau}(\overline{x}^{k}))}{u(\overline{x}^{k})-u_{\tau}(\overline{x}^{k})} \bigg/ G'(u(\overline{x}^{k}))-a(\overline{x}^{k}) \right) \left(G(u(\overline{x}^{k}))-G(u_{\tau}(\overline{x}^{k}))\right)\\
\leq &\ 0,
\end{aligned}
\end{equation}
where $u_{\tau}(\overline{x}^{k})<\xi(\overline{x}^{k})<u(\overline{x}^{k})$. Thus, \eqref{cxkvaCrk} can be simplified as follows
\begin{equation} \label{vaprk}
0<c_{d}\leq \varepsilon_{k}C_{1}r_{k}^{-sp}+C_{2}r_{k}^{p(1-s)}.
\end{equation}

When $k$ is sufficiently large,  we can set $\varepsilon_{k}=r_{k}^{t+sp}$, with $t>0$ fixed. We claim that the sequence $\{\overline{x}^{k}\}$  in \eqref{ovP} is still a maximizing sequence of $\omega_{\tau}(x)$ in $\R^{n}$. Indeed, let $C_{\Phi}:=\Phi_{r_{k}}(x^{k})-\Phi_{r_{k}}(x^{k}+r_{k}e)=1-e^{-\frac{1}{12}}$. From \eqref{otvkPr}, we have known that $\varepsilon_{k}$ satisfies $$\varepsilon_{k}\geq\frac{A-\omega_{\tau}(x^{k})}{C_{\Phi}},$$
So, if $\varepsilon_{k}=r_{k}^{t+sp}$, we choose that the sequence $\{r_{k}\}$ is such that
$$r^{sp+t}_{k}\geq\frac{A-\omega_{\tau}(x^{k})}{C_{\Phi}},$$
and
$$r_{k}\to 0,\,\,\mbox{as}\,\, k\rightarrow \infty.$$
This illustrates that the choice of the sequences $\{r_{k}\}$ and $\{\varepsilon_{k}\}$ depends on the maximizing sequence $\{x^{k}\}$.

Haven known that $|\overline{x}^{k}-x^{k}|\leq r_{k}$, and the continuity of $\omega_{\tau}(x)$, when $k$ is large enough, $r_{k}$ is small enough, we have that
for any $\delta>0$,
$$|\omega_{\tau}(x^{k})-\omega_{\tau}(\overline{x}^{k})|<\delta.$$
Setting $\delta=(1-C_{\Phi})r^{sp+t}_{k}$, it can imply that
\begin{equation}
\begin{aligned}
A-\omega_{\tau}(\overline{x}^{k})& =A-\omega_{\tau}(x^{k})+\omega_{\tau}(x^{k})-\omega_{\tau}(\overline{x}^{k})\\
&\ \leq C_{\Phi}\varepsilon_{k}+(1-C_{\Phi})r^{sp+t}_{k}\\
&\ \leq \varepsilon_{k}.
\end{aligned}
\end{equation}
Thus, we obtain that the claim hold. Actually, we may consider that $\{\overline{x}^{k}\}$ is the subsequence of $\{x^{k}\}$ in the sense of  $\overline{x}^{k}=x^{k+m}$ for some $m\in\Z$.

Hence, from \eqref{vaprk} we can imply that
$$0<c_{d}\leq C_{2}r_{k}^{t}+C_{3}r_{k}^{p(1-s)}\rightarrow 0,\,\,\,\mbox{as}\,\,k\rightarrow\infty,\,\,r_{k}\rightarrow 0,$$
which makes a contradiction. Therefore, for $\tau$ large enough, \eqref{ota} holds, i.e.
$$\omega_{\tau}\leq 0, \,\,\, \mbox{for}\,\, x\in \R^{n}.$$

$Step2$. Decrease $\tau$ continuously as long as \eqref{ota} is valid, we will prove that for any $\tau >0,$
\begin{equation} \label{anywtaul0}
\omega_{\tau}(x)\leq 0, \,\,\, \mbox{for}\,\,  x\in \R^{n}.
\end{equation}
 Define
$$\tau_{0}:= \inf\{\tau>0 |\, \omega_{\tau}(x)\leq 0, \forall x\in \R^{n}\},$$

Then, we claim that
\begin{equation}\label{tau00}
\tau_{0}=0.
\end{equation}
The claim \eqref{tau00} is equivalent to \eqref{anywtaul0}, which implies that the solution $u$ is monotone increasing in $x_{n}$- direction.

In the following, we prove the claim \eqref{tau00} by contradiction. Assume on the contrary that $\tau_{0}>0.$ By the definition of $\tau_{0}$, we have that
\begin{equation}\label{tau0leq0}
\omega_{\tau_{0}}(x)\leq 0, \,\,\, x\in \R^{n}.
\end{equation}
%In fact, assume on the contrary, there exists $\hat{x}$ (whether in finity or infinity) satisfying $$\omega_{\tau_{0}}(\hat{x})=A>0.$$
%We choose a sequence $\{\tau_{k}\}\subseteq \R^{n}_{+}$ such that $\tau_{k}\rightarrow\tau_{0}$ as $k\rightarrow\infty$ with $\omega_{\tau_{k}}(x)\leq 0$.
%For any $\varepsilon>0$, there exists $K$ such that $k>K$, we have
%$$0<A=\omega_{\tau_{0}}(\hat{x})<\omega_{\tau_{k}}(\hat{x})+\varepsilon,$$
%set $\varepsilon=\frac{1}{2}A$, it can implies that $0<\frac{1}{2}A\leq 0$, which is a confliction.
Then, there exists two cases that may occur as follows,
\begin{itemize}
 \item[$(1)$] $\sup_{x\in \Omega_{M}}\omega_{\tau_{0}}(x)=0.$
 \item[$(2)$] $\sup_{x\in \Omega_{M}}\omega_{\tau_{0}}(x)<0,$
 \end{itemize}
where $M$ is the same as in $Step1$. We will prove that the two cases can not exist, i.e. $\tau_{0}>0$ is not valid, which finishes the proof of the claim \eqref{tau00}.

$(1)$  $\sup_{x\in \Omega_{M}}\omega_{\tau_{0}}(x)=0.$

In the following, we will prove this case does not occur. We argue by contradiction, if $\sup_{x\in \Omega_{M}}\omega_{\tau_{0}}(x)=0$ is valid, then
there exists a sequence $\{x^{k}\}\subseteq\Omega_{M}$ such that
$$\omega_{\tau_{0}}(x^{k})\rightarrow0,\,\,\, \mbox{as}\,\, k\rightarrow\infty.$$
Taking the similar argument as in $Step1$, there exists $\overline{x}^{k}\in B_{r_{k}}(x^{k})$ satisfying
\begin{equation}\nonumber
\begin{aligned}
\omega_{\tau_{0}}(\overline{x}^{k})+\varepsilon_{k}\Phi_{r_{k}}(\overline{x}^{k})&=\max_{x\in B_{r_{k}}(x^{k})}(\omega_{\tau_{0}}(x)+\varepsilon_{k}\Phi_{r_{k}}(x))\\
&=\sup_{x\in \R^{n}}(\omega_{\tau_{0}}(x)+\varepsilon_{k}\Phi_{r_{k}}(x))\\
&\geq 0
\end{aligned}
\end{equation}
where the definition of $\Phi_{r_{k}}$  in \eqref{Phirk}. It follows that
\begin{equation}  \label{ovP0}
\begin{aligned}
0\geq\omega_{\tau_{0}}(\overline{x}^{k})&\geq \omega_{\tau_{0}}(x^{k})+\varepsilon_{k}\Phi_{r_{k}}(x^{k})-\varepsilon_{k}\Phi_{r_{k}}(\overline{x}^{k})\\
&\geq \omega_{\tau_{0}}(x^{k})\rightarrow 0,\,\,\, \mbox{as}\,\, k\rightarrow \infty,
\end{aligned}
\end{equation}
which implies that
\begin{equation}\label{ometau0olxk0}
\omega_{\tau_{0}}(\overline{x}^{k})\rightarrow 0,\,\,\, \mbox{as}\,\,k\rightarrow\infty.
\end{equation}
The same calculation and estimation as in $Step1$, we obtain that
\begin{equation}\nonumber
\begin{aligned}
C_{s,p}\int_{\R^{n}\backslash B_{2r_{k}}(x^{k})}\frac{G(-\omega_{\tau_{0}}(y))}{|x^{k}-y|^{n+sp}}dy & \leq \varepsilon_{k}C_{1}r_{k}^{-sp}+C_{2}r_{k}^{p(1-s)}+ f(u(\overline{x}^{k}))-f(u_{\tau_{0}}(\overline{x}^{k}))\\
& -a(\overline{x}^{k})\left(|u|^{p-2}u(\overline{x}^{k})-|u_{\tau_{0}}|^{p-2}u_{\tau_{0}}(\overline{x}^{k})\right).
\end{aligned}
\end{equation}
The term $C_{s,p}\int_{\R^{n}\backslash B_{2r_{k}}(x^{k})}\frac{G(-\omega_{\tau_{0}}(y))}{|x^{k}-y|^{n+sp}}dy$ is derived from
\begin{equation}\nonumber
\begin{aligned}
I_{2}&=\int_{\R^{n}\backslash B_{2r_{k}}(x^{k})}  \frac{G\left(u(\overline{x}^{k})+\varepsilon_{k}\Phi_{r_{k}}(\overline{x}^{k})-u(y)-\varepsilon_{k}\Phi_{r_{k}}(y)\right)-G\left(u_{\tau_{0}}(\overline{x}^{k})-u_{\tau_{0}}(y)\right)}{|\overline{x}^{k}-y|^{n+sp}}dy\\
& \geq 2^{2-p}C_{n,s,p}\int_{\R^{n}\backslash B_{2r_{k}}(x^{k})}\frac{G\left(\omega_{\tau_{0}}(\overline{x}^{k})+\varepsilon_{k}\Phi_{r_{k}}(\overline{x}^{k})-\omega_{\tau_{0}}(y)\right)}{|\overline{x}^{k}-y|^{n+sp}}dy\\
& \geq C_{s,p}\int_{\R^{n}\backslash B_{2r_{k}}(x^{k})}\frac{G(-\omega_{\tau_{0}}(y))}{|x^{k}-y|^{n+sp}}dy,
\end{aligned}
\end{equation}
where $C_{s,p}$ is a positive constant with depending on $s$ and $p$. Applying the fact that \eqref{tau0leq0}, we have that for any $k>0$
$$-\omega_{\tau_{0}}(y)\geq 0, \,\,\, y\in \R^{n}\backslash B_{2r_{k}}(x^{k}),$$
which implies that
\begin{equation} \label{ioyvd}
\begin{aligned}
0 & \leq C_{s,p}\int_{\R^{n}\backslash B_{2r_{k}}(x^{k})}\frac{G(-\omega_{\tau_{0}}(y))}{|x^{k}-y|^{n+sp}}dy\\
  &\leq \varepsilon_{k}C_{1}r_{k}^{-sp}+C_{2}r_{k}^{p(1-s)}+
  f(u(\overline{x}^{k}))-f(u_{\tau_{0}}(\overline{x}^{k}))\\
  &\ \quad -a(\overline{x}^{k})\left(|u|^{p-2}u(\overline{x}^{k})-|u_{\tau_{0}}|^{p-2}u_{\tau_{0}}(\overline{x}^{k})\right).
\end{aligned}
\end{equation}

Having knowing that $\omega_{\tau_{0}}(\overline{x}^{k})\rightarrow 0$, and when $k$ is large, we set $\varepsilon_{k}=r_{k}^{t+sp}$ as in $Step1$, the right hand side of \eqref{ioyvd} gives that
\begin{equation} \label{vCdft0}
\begin{aligned}
&C_{1}r_{k}^{t}+C_{2}r_{k}^{p(1-s)}+ f(u(\overline{x}^{k}))-f(u_{\tau_{0}}(\overline{x}^{k}))\\
&\ -a(\overline{x}^{k})\left(|u|^{p-2}u(\overline{x}^{k})-|u_{\tau_{0}}|^{p-2}u_{\tau_{0}}(\overline{x}^{k})\right)\rightarrow0,\,\,k\rightarrow \infty.
\end{aligned}
\end{equation}
Thus, from \eqref{ioyvd} we obtain that when $k\rightarrow \infty,$
\begin{equation} \label{Brkomta00}
\begin{aligned}
0 &\ \leq C_{s,p}\int_{\R^{n}\backslash B_{2r_{k}}(x^{k})}\frac{G(-\omega_{\tau_{0}}(y))}{|x^{k}-y|^{n+sp}}dy\\
&\ = C_{s,p}\int_{\R^{n}\backslash B_{2r_{k}}(0)}\frac{G(-\omega_{\tau_{0}}(y+x^{k}))}{|y|^{n+sp}}dy\\
&\ \rightarrow 0.
\end{aligned}
\end{equation}

 Here we give a claim that $G(-\omega_{\tau_{0}}(y+x^{k}))\rightarrow 0\,\, \mbox{for}\,\,k\rightarrow\infty,$ i.e., for any $\varepsilon>0$, there exists $K$ large such that for $k>K$, $$G(-\omega_{\tau_{0}}(y+x^{k}))<\varepsilon, \,\,\forall \,y\in B^{c}_{2r_{k}}(0).$$
In fact, we assume that the claim is not true. Hence, there exists $\delta>0$, for any $K>0$, then there must exist $k$ such that $k>K$ and $y_{0}\in B^{c}_{2r_{k}}(0)$, the following holds that
$$G(-\omega_{\tau_{0}}(y_{0}+x^{k}))>\delta.$$
Since the continuity of $\omega_{\tau_{0}}$ and $G$, we have that there exists a neighbourhood set $U_{y_{0}}^{\varepsilon_{0}}$ of $y_{0}$ with non-zero measures satisfying $$G(-\omega_{\tau_{0}}(y+x^{k}))>\frac{\delta}{2}\,\,\,\,\, \mbox{for}\,\, y\in U_{y_{0}}^{\varepsilon_{0}},$$
where $\varepsilon_{0}>0$ is fixed, only depending on $\delta$. Now, we consider that there will happen the following two cases. One case is that $dist(U_{y_{0}}^{\varepsilon_{0}},0 )<1$, then we have that
\begin{equation}\nonumber
\begin{aligned}
&\ C_{s,p}\int_{ B^{c}_{2r_{k}}(0)}\frac{G(-\omega_{\tau_{0}}(y+x^{k}))}{|y|^{n+sp}}dy\\
\geq &\ C_{s,p}\int_{ B^{c}_{2r_{k}}(0)\cap U_{y_{0}}^{\varepsilon_{0}}} \frac{G(-\omega_{\tau_{0}}(y+x^{k}))}{|y|^{n+sp}}dy\\
\geq &\ C_{s,p}\int_{ B^{c}_{2r_{k}}(0)\cap U_{y_{0}}^{\varepsilon_{0}}}\frac{\delta}{2}\\
\geq &\ C_{s,p,\delta,\varepsilon_{0}}>0,
\end{aligned}
\end{equation}
where the constant $C_{s,p,\delta,\varepsilon_{0}}$ only depends on $s,\,p,\,\delta$ and $\varepsilon_{0}$. However, this contradicts with \eqref{Brkomta00} as $k\rightarrow \infty$. In this case, we can obtain the above claim.
The other case is that when $dist(U_{y_{0}}^{\varepsilon_{0}},0)\geq 1$, we will imply that
\begin{equation}\nonumber
\begin{aligned}
&\ C_{s,p}\int_{ B^{c}_{2r_{k}}(0)}\frac{G(-\omega_{\tau_{0}}(y+x^{k}))}{|y|^{n+sp}}dy\\
\geq &\ C_{s,p}\int_{ B^{c}_{2r_{k}}(0)\cap U_{y_{0}}^{\varepsilon_{0}}} \frac{G(-\omega_{\tau_{0}}(y+x^{k}))}{|y|^{n+sp}}dy\\
\geq &\ C_{s,p}\int_{ B^{c}_{2r_{k}}(0)\cap U_{y_{0}}^{\varepsilon_{0}}}\frac{\delta/2}{|y|^{n+sp}}dy\\
\geq &\ C'_{s,p,\delta,\varepsilon_{0}}>0,
\end{aligned}
\end{equation}
where the constant $C'_{s,p,\delta,\varepsilon_{0}}$ only depends on $s,\,p,\,\delta$ and $\varepsilon_{0}$. There exists a contradiction with \eqref{Brkomta00} as $k\rightarrow \infty$ in the second case.

Combing the two cases, we obtain the claim that
 $$G(-\omega_{\tau_{0}}(y+x^{k}))\rightarrow 0, \quad \,\,\mbox{for}\,\,k\rightarrow\infty.$$
 Denote $u^{k}(y):= u(y+x^{k})$, having known the fact that $u^{k}$ is uniform H\"older continuous in $\R$, so $u^{k}$ is equip-continuous. By the Arzel$\grave{a}$-Ascoli Theorem, together with the above conclusion, we derive that there exists $u^{\infty}$ such that
\begin{equation}\nonumber
\begin{aligned}
\omega_{\tau_{0}}(y+x^{k})=u^{k}(y)-u^{k}(y+\tau_{0}e_{n})&\ \rightarrow u^{\infty}(y)-u^{\infty}(y+\tau_{0}e_{n})\\
&\ = 0, \,\,\quad \mbox{as}\,\,k\rightarrow\infty,
\end{aligned}
\end{equation}
in the words, that is
\begin{equation}\label{infty0}
u^{\infty}(y)\equiv u^{\infty}(y+\tau_{0}e_{n}),\,\quad \forall\,\, y\in \bigcup^{\infty}_{k=1}B^{c}_{2r_{k}}(0).
\end{equation}
Recall the sequence $\{x^{k}\}\subseteq\Omega_{M}$, and $u(x)\equiv 0$ in $\R \backslash \Omega$ while $u(x)>0$ in $\Omega$, by \eqref{infty0} we can imply that there exists $x^{0}\in \R \backslash \Omega$ such that
\begin{equation}\label{uytau00}
0=u^{\infty}(x^{0})=u^{\infty}(x^{0}+\tau_{0}e_{n})=u^{\infty}(x^{0}+\tau_{0}e_{n})=\cdots=u^{\infty}(x^{0}+k\tau_{0}e_{n}).
\end{equation}
From Theorem \ref{thm2-2} we have known that
\begin{equation}\label{uyoktau}
u^{\infty}(x^{0}+k\tau_{0}e_{n})\rightarrow \mu, \,\,\,\mbox{as}\, \,k\rightarrow \infty.
\end{equation}
Obviously, the equality \eqref{uytau00} contradicts \eqref{uyoktau}. Thus, we conclude that the case $(1)$ is not valid.

$(2)$ $\sup_{x\in \Omega_{M}}\omega_{\tau_{0}}(x)<0$.

Now, we prove that the case $(2)$ does not occur. On the contrary, we assume that $\sup_{x\in \Omega_{M}}\omega_{\tau_{0}}(x)<0$ is valid. Due to the continuity of $\omega$, for any $\eta\in (0,\tau_{0})$ small sufficiently, we obtain that
\begin{equation}
\sup_{x\in\Omega_{M}}\omega_{\tau}(x)\leq 0,\,\,\,0\leq\tau=\tau_{0}-\eta<\tau_{0}. \label{soteM}
\end{equation}
Suppose that $\sup_{x\in \R^{n}}\omega_{\tau}(x)>0,$ i.e., there exists a constant $A_{1}>0$ such that
$$\sup_{x\in \R^{n}}\omega_{\tau}(x)=A_{1}.$$
Then there exists a sequence $\{x^{k}\}\subseteq \R^{n}$ such that\\
$$\omega_{\tau}(x^{k})\rightarrow A_{1}, \,\,\, \mbox{as}\,\,k\rightarrow \infty.$$
Since $u=0$ in $\R^{n}\backslash\Omega$, we have that $\omega_{\tau}(x)\leq 0\,\,\,\mbox{for}\,\, x\in \R^{n}\backslash\Omega$. By Proposition \ref{thm2-2} and \eqref{soteM}, there exists a  constant $M_{1}$  such that as $ M_{1}>M >0$
$$\{x^{k}\}\subseteq \Omega_{M_{1}}\backslash \Omega_{M}.$$
And from the choice of $M$ in $Step1$, we also have that for any $x\in \Omega_{M_{1}}\backslash\Omega_{M}$,
$$u(x),\,u_{\tau}(x)\in (t_{1},\mu).$$

Taking the same arguments as in $Step1$ and the case $(1)$, we can also obtain the conclusion that
$$\omega_{\tau}(x)\leq 0, \,\,\, x\in \R^{n},$$
which contradicts the definition of $\tau_{0}$. Hence, the case $(2)$ can not happen.

To sum the above two cases up, we obtain that the claim \eqref{tau00} holds, i.e.,
$$\tau_{0}=0.$$
Thus, \eqref{anywtaul0} is also proved, i.e., for any $\tau>0$,
$$\omega_{\tau}(x)\leq 0, \,\,\,x\in\R^{n}.$$

Moreover, if $\sup_{x\in \Omega_{M}}\omega_{\tau}(x)=0$ for any $\tau>0$, recall the proof of the case $(1)$, we have known that this case does not happen. Therefore, we have that
$$\sup_{x\in \Omega_{M}}\omega_{\tau}(x)<0, \,\,\,\,\mbox{for}\,\,\tau>0.$$
This indicates that $u$ is strictly monotone increasing in $x_{n}-$direction. The proof of Theorem \ref{Them2} is completed.
\end{proof}

\textbf{Acknowledgments}

The authors thank sincerely to Professor Wenxiong Chen for valuable discussions and suggestions on the topic of this paper. The authors are grateful to the referees for their careful reading and valuable comments.

\end{document}